\documentclass[11pt,reqno,letterpaper]{amsart}

\usepackage[margin=1in]{geometry}
\usepackage[T1]{fontenc}
\usepackage{amsmath,amssymb,amsthm,mathtools}
\usepackage{tikz}
\usetikzlibrary{arrows.meta,decorations.pathreplacing,positioning,calc}
\usepackage[labelfont=bf]{caption}
\usepackage{aliascnt}
\usepackage[colorlinks=true,allcolors=blue,hypertexnames=false]{hyperref}
\numberwithin{equation}{section}

\newtheorem{theorem}{Theorem}[section]

\newaliascnt{proposition}{theorem}
\newtheorem{proposition}[proposition]{Proposition}
\aliascntresetthe{proposition}

\newaliascnt{lemma}{theorem}
\newtheorem{lemma}[lemma]{Lemma}
\aliascntresetthe{lemma}

\newaliascnt{remark}{theorem}
\newtheorem{remark}[remark]{Remark}
\aliascntresetthe{remark}

\usepackage[noabbrev,capitalize,nameinlink]{cleveref}
\crefname{theorem}{Theorem}{Theorems}
\crefname{proposition}{Proposition}{Propositions}
\crefname{lemma}{Lemma}{Lemmas}
\crefname{remark}{Remark}{Remarks}

\newcommand{\Z}{\mathbb Z}

\newcommand{\eps}{\varepsilon}
\newcommand{\mc}{\mathcal}
\newcommand{\mb}{\mathbb}
\newcommand{\abs}[1]{\left|#1\right|}
\newcommand{\set}[1]{\left\{#1\right\}}
\newcommand{\paren}[1]{\left(#1\right)}
\newcommand{\ind}{\mathbf 1}

\title{The Size of the Spanning-Tree Spectrum of Simple Graphs}

\author{Vishesh Jain}
\address{Department of Mathematics, Statistics, and Computer Science, University of Illinois Chicago, Chicago, IL 60607, USA}
\email{visheshj@uic.edu}

\begin{document}

\begin{abstract}
For a graph $G$, let $\tau(G)$ denote the number of spanning trees. We show that for every fixed $0 < c < 1/4$, the number of distinct values of $\tau(G)$, as $G$ ranges over simple graphs on $n$ vertices, is at least $\exp(c n \log n)$ for all sufficiently large $n$. This is optimal up to the choice of the constant $c$ and resolves a conjecture of Chan-Kontorovich-Pak regarding a problem of Sedl\'a\v{c}ek from the late 1960s. 
\end{abstract}

\maketitle

\section{Introduction}\label{sec:introduction}

For a finite graph \(G\), let \(\tau(G)\) denote the number of spanning trees of
\(G\). For a family of simple graphs $\mc{G}$, we define its spanning-tree spectrum to be
\[\mc{T}(\mc{G}) := \set{\tau(G): G \in \mc{G}}.\]
In this paper, we study the spanning-tree spectrum of the family of simple graphs on $n$ vertices:
\[
\mc T_n:=\set{\tau(G): G \text{ is a simple graph on } n \text{ vertices}}.
\]
By Cayley's formula, every simple graph on \(n\)
vertices has at most \(n^{n-2}\) spanning trees, and therefore
\[
\abs{\mc T_n}\le n^{n-2}+1=\exp\bigl((1+o(1))n\log n\bigr).
\]
The main result of this paper gives a lower bound of the same shape.

\begin{theorem}\label{thm:main}
For every fixed \(0<c<1/4\), and all sufficiently large \(n\),
\[
|\mc{T}_n|
    \ge \exp(cn\log n).
\]
\end{theorem}

\subsection{Background}

The problem of understanding the size of the spanning-tree spectrum
goes back to Sedl\'a\v{c}ek, who initiated the study of spanning-tree spectra
for several natural graph classes, including all simple graphs, planar graphs,
and regular graphs \cite{sedlacek66,sedlacek69,sedlacek70minimal,sedlacek70regular} in the late 1960s.
The question was also recorded early in Moon's monograph on labelled trees
\cite{moon}.  Until recently, however, even proving exponential growth for the
spanning-tree spectrum of planar graphs was open.

Let \(\mc G_n^{\rm pl}\) denote the family of connected simple planar graphs on
\(n\) vertices. Since a planar graph on $n\geq 3$ vertices has at most $3n-6$ edges, $\tau(G) \leq \binom{3n-6}{n-1}$ for \(G\in \mc G_n^{\rm pl}\). Thus, the spanning-tree spectrum for planar graphs grows at most exponentially in $n$; the current best known upper bound is $5.2852^n$ due to Buchin--Schulz \cite{buchin-schulz}.  On the lower-bound side, Sedl\'a\v{c}ek \cite{sedlacek69}
proved a polynomial lower bound, namely
\[
\abs{\mc T(\mc G_n^{\rm pl})}=\Omega(n^2),
\]
and Azarija \cite{azarija} improved this to
\[
\abs{\mc T(\mc G_n^{\rm pl})}\ge
2^{\Omega(\sqrt{n/\log n})}.
\]
A result of Stong \cite{stong} on the dual problem of constructing small
graphs with a prescribed number of spanning trees implies the stronger bound
\[
\abs{\mc T(\mc G_n^{\rm pl})}\ge 2^{\Omega(n^{2/3})}.
\]
In a remarkable paper, the exponential growth of the spanning-tree spectrum of planar graphs was finally settled in the affirmative by Chan--Kontorovich--Pak \cite{ckp}, who proved that
\[
\abs{\mc T(\mc G_n^{\rm pl})}\ge 1.1103^n
\]
for all sufficiently large \(n\). 
Their argument combines a
deletion--contraction construction for marked planar graphs, continued-fraction
machinery, and Diophantine input related to Zaremba's conjecture.

Subsequently, Alon--Buci\'c--Gishboliner \cite{abg} gave a short,
purely combinatorial proof of the stronger bound
\[
\abs{\mc T(\mc G_n^{\rm pl})}\ge 1.55^n.
\]
In the same work, they also proved that, for every fixed \(k\ge3\), the
spanning-tree spectrum of connected \(k\)-regular graphs on \(n\) vertices has
exponential size whenever \(kn\) is even, determining the correct growth order
in another problem first raised by Sedl\'a\v{c}ek.

\medskip

The unrestricted simple-graph case is qualitatively different.  For planar
graphs, and more generally for graph families with bounded average degree, the
number of spanning trees of an \(n\)-vertex graph is at most exponential in
\(n\).  For all simple graphs, however, Cayley's formula leaves open the much
larger scale
\[
\abs{\mc T_n}\le n^{n-2}+1=\exp((1+o(1))n\log n).
\]
Until now, no lower bound on this larger scale was known.  As
Alon--Buci\'c--Gishboliner \cite{abg} point out, it was ``somewhat remarkable''
that, historically, the essentially best known lower bounds for the unrestricted
spanning-tree spectrum all came from planar constructions, despite the much
greater freedom available in arbitrary simple graphs.  In particular, any
superexponential lower bound for \(\abs{\mc T_n}\) must use genuinely non-planar
degrees of freedom.

This makes it natural to conjecture that the unrestricted spectrum has the same
\(n\log n\) order as the Cayley upper bound.  Chan--Kontorovich--Pak made this
explicit in \cite[Conjecture~5.2]{ckp}, conjecturing that
\[
\abs{\mc T_n}=\exp(\Omega(n\log n)).
\]
\cref{thm:main} proves this conjecture. 

\medskip

It is useful to compare this with the corresponding range problem for
\(0\)-\(1\) determinants.  Let \(\sigma(n)\) be the set of positive values
attained by determinants of \(n\times n\) matrices with entries in
\(\{0,1\}\), and let \(b(n)\) be the largest integer such that
\[
\{1,2,\dots,b(n)\}\subseteq \sigma(n).
\]
Hadamard's inequality gives the upper bound
\[
|\sigma(n)|,\ b(n)\le
\exp\bigl((1/2+o(1))n\log n\bigr).
\]
In a recent breakthrough, Shitov \cite{shitov} proved that this upper bound is sharp up to the lower-order
term in the exponent: for all sufficiently large \(n\),
\[
\{0,1,\dots,\lfloor
\exp(0.5\,n\log n-7n\log\log n)
\rfloor\}
\]
is contained in the set of determinants of \(n\times n\) \(0\)-\(1\) matrices. Thus the \(0\)-\(1\) determinant problem has been solved not only for
the size of the range, but also in the stronger interval sense, and with the correct constant in the exponent.

By Kirchhoff's Matrix--Tree Theorem, spanning-tree counts are also determinant
values, but of a much more constrained kind: they are determinants of Laplacian
cofactors of simple graphs. In this sense, \cref{thm:main} may be viewed as a
Matrix--Tree analogue of the \(0\)-\(1\) determinant range problem, in
that it proves the correct \(n\log n\) order for the size of the range.  There
are, however, two important respects in which the graph problem remains less
understood.  First, our constant \(1/4\) is not expected to be sharp; the
trivial Cayley upper bound allows the leading constant \(1\).  Second, unlike
Shitov's theorem for \(0\)-\(1\) determinants, our construction does not produce
a long initial interval of spanning-tree counts.

The interval question is closely related to the dual form of Sedl\'a\v{c}ek's
problem.  If \(\alpha'(t)\) denotes the least number of vertices in a simple
graph with exactly \(t\) spanning trees, then proving that
\[
\{1,2,\dots,\exp(cn\log n)\}\subseteq \mc T_n
\]
would be equivalent to the uniform bound
\[
\alpha'(t)\le (1/c+o(1))\frac{\log t}{\log\log t}
\]
throughout that range.  Sedl\'a\v{c}ek initiated the study of this dual problem, and
it has since been investigated by Nebesk\'y, Azarija--\v{S}krekovski, Stong, and
Chan--Kontorovich--Pak
\cite{sedlacek70minimal,Nebesky1973,AzarijaSkrekovski2013,stong,ckp}.  We do not
address this stronger interval/dual problem here.

\subsection{Overview of the proof}
\label{sec:overview}

The proof is short, and the details are given in the next few sections. Here, we isolate the main ideas and compare them
with previous approaches.

\medskip

A central point of departure, compared to previous works, is that we do not try to construct graphs whose spanning-tree
counts are directly distinguishable.  Instead, for each word
\[
\boldsymbol w=(w_1,\dots,w_m)\in\{2,\dots,q+1\}^m
\]
we construct a simple graph \(G_{\boldsymbol w}\) such that \(\tau(G_{\boldsymbol w})\)
is divisible by
\[
D_{\boldsymbol w}
=
K_m(w_1,\dots,w_m)K_{m-1}(w_1,\dots,w_{m-1}),
\]
where \(K_r\) is a continuant determinant. Using the continued fraction algorithm, the two factors in
\(D_{\boldsymbol w}\), taken as an ordered pair, recover the word
\(\boldsymbol w\).  Thus, by the standard divisor-function bound, the divisors \(D_{\boldsymbol w}\) themselves take many
different values. A second
divisor-function bound then shows that these many divisors cannot be
supported on too few values of \(\tau(G_{\boldsymbol w})\). This kind of argument is familiar in elementary multiplicative number theory.
A basic example is the multiplication-table problem: one asks how many distinct
products \(ab\) arise from a given set of pairs, and a first lower bound is
obtained by controlling how many pairs can produce the same product.  Divisor
bounds provide the basic multiplicity control; see, for instance, Ford's work
on the multiplication-table problem \cite{ford}.  In our setting, the role of
the products \(ab\) is played by the continuant products
\(D_{\boldsymbol w}\).  The reconstruction lemma (\cref{lem:continuant-reconstruction}) shows that, once an ordered
factorization
\[
D_{\boldsymbol w}
=
K_m(w_1,\dots,w_m)K_{m-1}(w_1,\dots,w_{m-1})
\]
is fixed, the word \(\boldsymbol w\) is fixed.  Thus the only remaining
multiplicity comes from the number of possible factorizations of
\(D_{\boldsymbol w}\), which is controlled by the divisor function.  The novelty
is to create such a large supply of distinguishable divisors inside
spanning-tree counts by manufacturing them as factors of
Matrix--Tree cofactors.

\medskip

We now briefly discuss how the divisors are produced.  The guiding object is a
rooted multigraph, presented as a warm-up in
\cref{sec:multigraph-warmup}.  It consists of two disjoint paths, on \(m\) and
\(m-1\) vertices, together with a distinguished root \(\rho\); each path vertex
is joined to \(\rho\) by a prescribed number of parallel edges.  Choosing these
multiplicities appropriately makes the two path blocks in the reduced Laplacian
equal to the tridiagonal matrices defining
\[
K_m(w_1,\dots,w_m)
\qquad\text{and}\qquad
K_{m-1}(w_1,\dots,w_{m-1}).
\]
Thus the warm-up construction realizes the desired product of continuants
directly.

The simple-graph construction in \cref{sec:graph-family} is obtained by
modifying this multigraph model.  The parallel root-edges are replaced by edges
to a common set of anchor vertices.  This removes multiple edges, but it also
couples the path blocks to the anchor block after the root is deleted. To recover a determinant factor, we duplicate each path.  The graph then has an
involution exchanging the two copies of that path.  Passing to the symmetric and
antisymmetric subspaces of such an involution is a standard device in spectral
graph theory, see for instance
\cite{butler-edge-coverings, butler-twins}.  In the present argument, this symmetry has an arithmetic role: inside the
Matrix--Tree cofactor, the antisymmetric coordinates decouple from the anchor
block and contribute exactly the desired path determinant.  The exact
linear-algebraic identity is \cref{lem:two-copy}, and the resulting
divisibility statement is \cref{prop:simple-graph-determinant-factor}.

\medskip

It remains to explain where the superexponential growth enters.  The
constructions of Chan--Kontorovich--Pak and Alon--Buci\'c--Gishboliner are
organized around deletion--contraction data associated with a marked edge, i.e.~
\[
\bigl(\tau(G-e),\tau(G/e)\bigr),
\]
and local graph operations which transform this two-entry vector by elementary
\(2\times2\) integer matrices.  In Chan--Kontorovich--Pak, this matrix calculus
is linked to continued fractions and then to Diophantine input related to
Zaremba's conjecture.  Our continued-fraction structure arises in a different
place.  It comes from the continuant recurrence
\[
K_r=w_rK_{r-1}-K_{r-2},
\qquad\text{equivalently}\qquad
\binom{K_r}{K_{r-1}}
=
\begin{pmatrix}
w_r&-1\\
1&0
\end{pmatrix}
\binom{K_{r-1}}{K_{r-2}}.
\]
Thus the common feature is a two-dimensional matrix recurrence, but the
quantities being propagated are different;  in \cite{abg,ckp} they are
deletion--contraction counts of marked graphs, while here they are consecutive
minors of a tridiagonal block inside a Matrix--Tree cofactor.

This difference is also responsible for the larger scale of the construction. In the deletion--contraction framework of the previous constructions, the graph operations are
local, and increasing the number of choices available at a given step requires a
larger local gadget at that step.  Thus the amount of information encoded grows
linearly with the number of vertices, leading naturally to exponentially many
possibilities. Our construction avoids this local cost by sharing the large part of the
alphabet.  We introduce \(q\) anchor vertices once, and then each of the \(m\)
path positions may choose any one of \(q\) possible weights
\(w_i\in\{2,\dots,q+1\}\).  Thus the same anchor set supports \(q\) choices at
each of \(m\) positions.  Taking $q = \lfloor \eps m \rfloor $, this produces
\[
q^m=\exp((1+o(1))m\log m)
\]
candidate continuant divisors while using only
\[
4m+q-1=(4+\eps)m +O(1)
\]
vertices.  

\medskip

Finally, the result is related in spirit to Shitov's recent theorem on
\(0\)-\(1\) determinants \cite{shitov}.  Shitov's construction first realizes the required range using determinants of auxiliary integer-valued matrices, and then uses determinant-preserving gadgets to convert these
matrices into \(0\)-\(1\) matrices.  In the Matrix--Tree setting, the available
determinants are far more constrained, since they must arise as Laplacian cofactors of
simple graphs, with the symmetry, sign pattern, and sparsity imposed by graph
structure.  We therefore do not try to encode the determinant value itself.
Instead, we only try to construct many determinant factors inside such cofactors, and then
use divisor-counting to force many distinct values of \(\tau\).

\subsection{Acknowledgments}
We thank Swee Hong Chan for his beautiful talk at the 10th Lake Michigan Workshop on Combinatorics which introduced us to the problem. We also thank Marcus Michelen, Huy Tuan Pham, Mehtaab Sawhney, and the participants of Desert Discrete Math Workshop 3 for useful discussions. ChatGPT 5.5 Pro drew our attention to the ``twin trick'' in spectral graph theory and assisted in preparing the manuscript. The author is partially supported by NSF grant DMS-2237646.

\section{Preliminaries}\label{sec:prelims}

We first recall two standard facts that will be used repeatedly.

\begin{theorem}[Kirchhoff's Matrix-Tree Theorem; see~{\cite[Chapter 6]{biggs}}]\label{thm:matrix-tree}
Let $G$ be a finite connected multigraph with Laplacian $L(G)$.  For any
vertex $v$, let $L(G)^{(v)}$ be the matrix obtained from $L(G)$ by deleting the
row and column indexed by $v$.  Then
\[
\tau(G)=\det L(G)^{(v)}.
\]
\end{theorem}

\begin{lemma}[Divisor bound; see {\cite{hardy-wright}}]\label{lem:divisor-bound}
For every $\eta>0$ there exists $X_0(\eta) \in \mb{N}$ such that, whenever $X\ge X_0(\eta)$
and $1\le M\le X$,
\[
\sigma_0(M)\le X^\eta,
\]
where $\sigma_0(M)$ denotes the number of positive divisors of $M$.
\end{lemma}

We now turn to the continuants used in the construction.  For $r \in \mb{Z}_{\geq 0}$ and $x_1,\dots, x_r \in \mb{Z}$, define $K_r(x_1,\dots, x_r)$ recursively by
\[
K_0=1,\qquad K_1(x_1)=x_1,
\]
and, for $i\ge2$,
\begin{equation}
\label{eq:continuant}
K_i(x_1,\dots,x_i)=x_i K_{i-1}(x_1,\dots,x_{i-1})-K_{i-2}(x_1,\dots,x_{i-2}).
\end{equation}
Equivalently,
\[
K_r(x_1,\dots,x_r)=\det T_r(x_1,\dots,x_r),
\]
where
\[
T_r(x_1,\dots,x_r)=
\begin{pmatrix}
x_1&-1&0&\cdots&0\\
-1&x_2&-1&\ddots&\vdots\\
0&-1&x_3&\ddots&0\\
\vdots&\ddots&\ddots&\ddots&-1\\
0&\cdots&0&-1&x_r
\end{pmatrix}.
\]
\begin{remark}
    It is more common to define continuants with a positive sign in the recurrence \eqref{eq:continuant}. This corresponds, for example, to changing one of the two off-diagonals of \(T_r\) from \(-1\) to \(+1\).
\end{remark}

The next lemma, which is classical, records the continued-fraction interpretation of continuants and the resulting reconstruction algorithm. Due to our sign convention, we will work with negative continued fractions. Recall that for $r\geq 1$ and integers $x_1,\dots, x_r$, the (finite) minus continued fraction $[x_r,\dots, x_1]_-$ is defined recursively by
\[
[x_1]_-:=x_1,\qquad
[x_i,\dots,x_1]_-:=x_i-\frac{1}{[x_{i-1},\dots,x_1]_-}
\quad (i\ge2).
\]

\begin{lemma}\label{lem:continuant-reconstruction}
Let $r\ge1$, and let $x_1,\dots,x_r$ be integers with $x_i\ge2$.  
Then, for all $1\leq i \leq r$, $K_i(x_1,\dots, x_i) \in \mb{Z}_{> 0}$ and moreover,
\[
\frac{K_i(x_1,\dots,x_i)}{K_{i-1}(x_1,\dots, x_{i-1})}=[x_i,\dots,x_1]_-.
\]
In particular, the map \[ (x_1,\dots,x_r) \mapsto \paren{K_r(x_1,\dots,x_r),K_{r-1}(x_1,\dots,x_{r-1})} \] from $\set{(x_1,\dots,x_r)\in\Z^r:x_i\ge 2}$ to $\Z_{>0}^2$ is injective.
\end{lemma}

\begin{proof}
For convenience of notation, let 
\[
A_i=K_i(x_1,\dots,x_i),\qquad A_0=1,\qquad A_{-1}=0.
\]
We first show that
\begin{equation}\label{eq:continuant-monotonicity}
0\le A_{i-2}<A_{i-1}\qquad (i\ge1).
\end{equation}
For $i=1$, this is $0=A_{-1}<A_0=1$.  If
$0\le A_{i-2}<A_{i-1}$, then the continuant recurrence \eqref{eq:continuant} gives
\[
A_i=x_iA_{i-1}-A_{i-2}
    \ge 2A_{i-1}-A_{i-2}
    > A_{i-1},
\]
since $x_i\ge2$.  Thus \eqref{eq:continuant-monotonicity} follows by
induction.

The same recurrence gives, for every $i\ge1$,
\[
\frac{A_i}{A_{i-1}}
=
x_i-\frac{A_{i-2}}{A_{i-1}}.
\]
For $i=1$, this says $A_1/A_0=x_1=[x_1]_-$.  For $i\ge2$, we have
\[
\frac{A_i}{A_{i-1}}
=
x_i-\frac{1}{A_{i-1}/A_{i-2}}.
\]
By induction on $i$, this is exactly
\[
\frac{A_i}{A_{i-1}}=[x_i,\dots,x_1]_-.
\]

It remains only to spell out why this continued-fraction expansion is uniquely
recoverable from the final numerator-denominator pair.  From
\eqref{eq:continuant-monotonicity},
\[
0\le \frac{A_{i-2}}{A_{i-1}}<1,
\]
and hence
\[
x_i-1<\frac{A_i}{A_{i-1}}\le x_i.
\]
Therefore
\[
x_i=\left\lceil \frac{A_i}{A_{i-1}}\right\rceil.
\]
Once $x_i$ is known, $A_{i-2}$ is recovered by rearranging the
continuant recurrence:
\[
A_{i-2}=x_iA_{i-1}-A_i.
\]
Hence, starting from \((A_r,A_{r-1})\), the formulas
\[
x_i=\left\lceil\frac{A_i}{A_{i-1}}\right\rceil,
\qquad
A_{i-2}=x_iA_{i-1}-A_i
\]
recover \(x_i\) and then \(A_{i-2}\), for \(i=r,r-1,\dots,1\);
at the final step this gives \(A_{-1}=0\), as expected. 
\end{proof}

The following crude estimates on the continuants will suffice for our purpose.

\begin{lemma}\label{lem:continuant-product-bound}
Let $r\ge 1$ and $x_1,\dots,x_r$ be integers with $x_i\ge2$. Then,
\[
1\le K_r(x_1,\dots,x_r)\le\prod_{i=1}^r x_i.
\]
\end{lemma}

\begin{proof}
The lower bound was proved in \cref{lem:continuant-reconstruction}. For the upper bound, using the continuant recurrence \eqref{eq:continuant} and non-negativity of continuants, we have
\[
K_i(x_1,\dots, x_i) =x_iK_{i-1}(x_1,\dots, x_{i-1})-K_{i-2}(x_1,\dots, x_{i-2}) \le x_iK_{i-1}.
\]
Iterating this, and recalling that $K_0 = 1$, gives $K_r\le\prod_{i=1}^r x_i$.
\end{proof}

The next identity is the linear-algebraic form of the twin-copy idea: from two identical copies, one copy factors out while the other remains coupled to the exterior.

\begin{lemma}
\label{lem:two-copy}
Let \(s,t\ge 1\).  Let
\[
A\in \mb{Z}^{s\times s},\qquad
E\in \mb{Z}^{s \times t},\qquad
G\in \mb{Z}^{t \times s},\qquad
H\in \mb{Z}^{t \times t}.
\]
Then
\[
\det
\begin{pmatrix}
A&0&E\\
0&A&E\\
G&G&H
\end{pmatrix}
=
\det(A)\,
\det
\begin{pmatrix}
A&E\\
2G&H
\end{pmatrix}.
\]
\end{lemma}

\begin{proof}
Subtract the second block row from the first block row.  This gives
\[
\det\begin{pmatrix}
A&0&E\\
0&A&E\\
G&G&H
\end{pmatrix}
=
\det\begin{pmatrix}
A&-A&0\\
0&A&E\\
G&G&H
\end{pmatrix}.
\]
Then add the first block column to the second block column.  This gives
\[
\det\begin{pmatrix}
A&-A&0\\
0&A&E\\
G&G&H
\end{pmatrix}
=
\det\begin{pmatrix}
A&0&0\\
0&A&E\\
G&2G&H
\end{pmatrix}.
\]
Expanding along the first block row gives the claimed identity.  
\end{proof}

\section{Warm-up: a multigraph construction}\label{sec:multigraph-warmup}

As a warm-up, we begin by describing a multigraph version of our construction, which already contains most of the key ideas. In the next section, two modifications will turn this idea into a simple-graph construction. Our multigraphs will be parametrized by
\[
\boldsymbol{w}=(w_1,\dots,w_m)\in\set{2,3,\dots,q+1}^m,
\]
where $m\geq 3$ is half the number of vertices and $q\geq 1$ is a free parameter. In our application, we will take $q = \lfloor \eps m \rfloor $ for sufficiently small $\eps > 0$. 

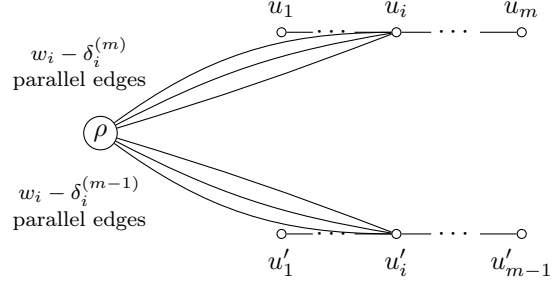
\begin{figure}[t]
\centering
\begin{tikzpicture}[scale=.92,every node/.style={font=\small}]
    \node[circle,draw,inner sep=1.6pt] (rho) at (0,0) {$\rho$};
    \node[font=\scriptsize,align=center] at (0,-.75){};

    \node[circle,draw,inner sep=1.2pt,label=above:{$u_1$}] (u1) at (2.6,1.45) {};
    \node at (3.35,1.45) {$\cdots$};
    \node[circle,draw,inner sep=1.2pt,label=above:{$u_i$}] (ui) at (4.25,1.45) {};
    \node at (5.1,1.45) {$\cdots$};
    \node[circle,draw,inner sep=1.2pt,label=above:{$u_m$}] (um) at (6.05,1.45) {};
    \draw (u1)--(3.05,1.45);
    \draw (3.65,1.45)--(ui);
    \draw (ui)--(4.75,1.45);
    \draw (5.45,1.45)--(um);

    \node[circle,draw,inner sep=1.2pt,label=below:{$u'_1$}] (up1) at (2.6,-1.45) {};
    \node at (3.35,-1.45) {$\cdots$};
    \node[circle,draw,inner sep=1.2pt,label=below:{$u'_i$}] (upi) at (4.25,-1.45) {};
    \node at (5.1,-1.45) {$\cdots$};
    \node[circle,draw,inner sep=1.2pt,label=below:{$u'_{m-1}$}] (upm) at (6.05,-1.45) {};
    \draw (up1)--(3.05,-1.45);
    \draw (3.65,-1.45)--(upi);
    \draw (upi)--(4.75,-1.45);
    \draw (5.45,-1.45)--(upm);

    \draw[bend left=18] (rho) to (ui);
    \draw[bend left=9] (rho) to (ui);
    \draw[bend right=2] (rho) to (ui);
    \node[font=\scriptsize,align=center] at (-0.3,1) {$w_i-\delta_i^{(m)}$\\parallel edges};

    \draw[bend right=18] (rho) to (upi);
    \draw[bend right=9] (rho) to (upi);
    \draw[bend left=2] (rho) to (upi);
    \node[font=\scriptsize,align=center] at (-0.3,-1) {$w_i-\delta_i^{(m-1)}$\\parallel edges};

    \node[font=\scriptsize,align=center] at (5.35,0) {};
\end{tikzpicture}
\caption{A schematic of the multigraph construction. All multi-edges involve the distinguished root $\rho$.}
\label{fig:multigraph-warmup}
\end{figure}

\medskip

Throughout, let \(P_r\) denote the path
\[
1-2-\cdots-r
\]
on vertex set \([r]\).  For \(i\in[r]\), write
\[
\delta_i^{(r)}:=\deg_{P_r}(i)
=\ind_{i>1}+\ind_{i<r}.
\]
We now construct \(H_{\boldsymbol w}\); see \cref{fig:multigraph-warmup}.
Let \(\rho\) be a distinguished root vertex, and take two paths, disjoint from each other
and from \(\rho\),
\[
u_1-\cdots-u_m,
    \qquad
u'_1-\cdots-u'_{m-1}.
\]
For \(1\le i\le m\), add \(w_i-\delta_i^{(m)}\) parallel edges between
\(u_i\) and \(\rho\).  For \(1\le i\le m-1\), add
\(w_i-\delta_i^{(m-1)}\) parallel edges between \(u'_i\) and \(\rho\).
These multiplicities are nonnegative since \(w_i\ge2\) and
\(\delta_i^{(r)}\le2\).  Moreover, each endpoint receives at least one edge
to \(\rho\), so both paths are attached to the root and the resulting
multigraph is connected.

\begin{proposition}\label{prop:multigraph-warmup}
Let $\boldsymbol w  = (w_1,\dots, w_m) \in \{2,3,\dots, q+1\}^m$. For the multigraph \(H_{\boldsymbol w}\) constructed above,
\[
\tau(H_{\boldsymbol w})
    =K_m(w_1,\dots,w_m)K_{m-1}(w_1,\dots,w_{m-1}).
\]
\end{proposition}

\begin{proof}
Let
\[
h_i=w_i-\delta_i^{(m)}\qquad (1\le i\le m),
\]
and
\[
h'_i=w_i-\delta_i^{(m-1)}\qquad (1\le i\le m-1),
\]
so that \(h_i\) is the number of edges between \(u_i\) and
\(\rho\) and similarly for \(h'_i\). Order the vertices as
\[
\rho,\ u_1,\dots,u_m,\ u'_1,\dots,u'_{m-1}.
\]
Let
\[
s=\sum_{i=1}^m h_i+\sum_{i=1}^{m-1}h'_i,
\qquad
h=(h_1,\dots,h_m)^T,
\qquad
h'=(h'_1,\dots,h'_{m-1})^T.
\]
Then the Laplacian of \(H_{\boldsymbol w}\) has block form
\[
L(H_{\boldsymbol w})=
\begin{pmatrix}
s & -h^T & -(h')^T\\
-h & T_m(w_1,\dots,w_m) & 0\\
-h' & 0 & T_{m-1}(w_1,\dots,w_{m-1})
\end{pmatrix}.
\]
Applying Kirchhoff's Matrix-Tree Theorem with the cofactor obtained by deleting the row and column indexed
by \(\rho\), we have 
\begin{align*}
\tau(H_{\boldsymbol w}) &= \det \begin{pmatrix}
 T_m(w_1,\dots,w_m) & 0\\
0 & T_{m-1}(w_1,\dots,w_{m-1})
\end{pmatrix}\\
&= \det T_m(w_1,\dots,w_m)
\det T_{m-1}(w_1,\dots,w_{m-1})\\
&= K_m(w_1,\dots,w_m)K_{m-1}(w_1,\dots,w_{m-1}),
\end{align*}
as claimed. 
\end{proof}

\section{The simple graph family}\label{sec:graph-family}

We now turn the multigraph construction into a simple graph construction.  There
are two changes.  First, the parallel edges from a path vertex to the root
\(\rho\) are replaced by edges to distinct ``anchor'' vertices \(a_1,\dots,a_q\),
each of which is adjacent to \(\rho\).  Second, for each of the two paths, we create a ``twin copy''. The reason for this is the following. In the
multigraph construction, deleting the root $\rho$ left the path blocks completely decoupled. However, in our simple graph construction, the path blocks are still coupled after deleting $\rho$ due to the anchor vertices, so that the determinant of the path block need not divide the determinant of the cofactor. However, as we saw in \cref{lem:two-copy}, two identical path blocks with identical anchor connections have a special block form, and this block form lets us factor off
one copy of the path determinant exactly.  

\medskip

We now proceed to the formal details. As before, fix integers \(m\ge3\) and \(q\ge1\), and let
\[
\boldsymbol{w}=(w_1,\dots,w_m)\in\set{2,3,\dots,q+1}^m.
\]
We will construct a simple connected graph 
\[
G_{\boldsymbol w}=G_{m,q,\boldsymbol w}
\]
on \[N = N_{m,q} = 4m + q-1\] 
vertices; see \cref{fig:simple-graph}. Begin with an anchor star: a root vertex \(\rho\), vertices
\(a_1,\dots,a_q\), and edges \(\rho a_t\) for \(1\le t\le q\).  Next take
four paths, all disjoint from the anchor star:
\[
u_1-\cdots-u_m,
    \qquad
v_1-\cdots-v_m,
\]
and
\[
u'_1-\cdots-u'_{m-1},
    \qquad
v'_1-\cdots-v'_{m-1}.
\]
For \(1\le i\le m\), set
\[
h_i=w_i-\delta_i^{(m)}.
\]
Join both \(u_i\) and \(v_i\) to the anchor vertices
\[
a_1,\dots,a_{h_i},
\]
with the convention that this segment is empty if \(h_i=0\).  Similarly, for
\(1\le i\le m-1\), set
\[
h'_i=w_i-\delta_i^{(m-1)},
\]
and join both \(u'_i\) and \(v'_i\) to
\[
a_1,\dots,a_{h'_i},
\]
again with the same convention if $h_i' = 0$. There
are no other edges. 

By construction, the graph is simple. Moreover, for each path endpoint \(\delta_i^{(r)}=1\), so
\(h_i\ge1\) and \(h'_i\ge1\). Thus every path is attached to the anchor
star, and the graph is connected.

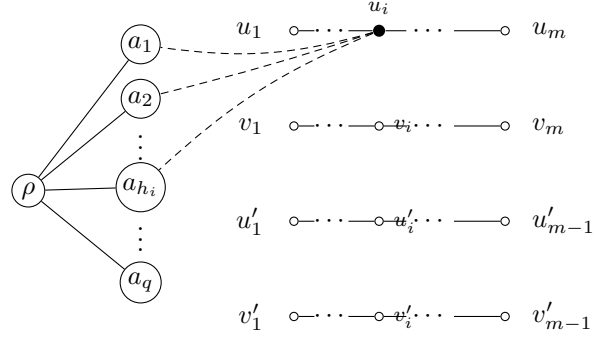
\begin{figure}[t]
\centering
\begin{tikzpicture}[scale=.9,every node/.style={font=\small}]
    \node[circle,draw,inner sep=1.5pt] (rho) at (0,0) {$\rho$};
    \node[circle,draw,inner sep=1.2pt] (a1) at (1.65,2.15) {$a_1$};
    \node[circle,draw,inner sep=1.2pt] (a2) at (1.65,1.35) {$a_2$};
    \node at (1.65,.75) {$\vdots$};
    \node[circle,draw,inner sep=1.2pt] (ahi) at (1.65,.05) {$a_{h_i}$};
    \node at (1.65,-.6) {$\vdots$};
    \node[circle,draw,inner sep=1.2pt] (aq) at (1.65,-1.35) {$a_q$};

    \draw (rho)--(a1);
    \draw (rho)--(a2);
    \draw (rho)--(ahi);
    \draw (rho)--(aq);
    \node[font=\scriptsize,align=center] at (.35,-2.0) {};

    \node[left] at (3.6,2.35) {$u_1$};
    \node[circle,draw,inner sep=1.1pt] (u1) at (3.9,2.35) {};
    \node at (4.45,2.35) {$\cdots$};
    \node[circle,fill,inner sep=1.6pt,label=above:{\scriptsize \(u_i\)}] (ui) at (5.15,2.35) {};
    \node at (5.9,2.35) {$\cdots$};
    \node[circle,draw,inner sep=1.1pt] (um) at (7.0,2.35) {};
    \node[right] at (7.25,2.35) {$u_m$};

    \draw (u1)--(4.2,2.35);
    \draw (4.7,2.35)--(ui);
    \draw (ui)--(5.55,2.35);
    \draw (6.25,2.35)--(um);

    \foreach \name/\y/\leftlab/\midlab/\rightlab in {
        v/.95/$v_1$/$v_i$/$v_m$,
        up/-.45/$u'_1$/$u'_i$/$u'_{m-1}$,
        vp/-1.85/$v'_1$/$v'_i$/$v'_{m-1}$
    }{
        \node[left] at (3.6,\y) {\leftlab};
        \node[circle,draw,inner sep=1.1pt] (\name 1) at (3.9,\y) {};
        \node at (4.45,\y) {$\cdots$};
        \node[circle,draw,inner sep=1.1pt,label=right:{\scriptsize \midlab}] (\name 2) at (5.15,\y) {};
        \node at (5.9,\y) {$\cdots$};
        \node[circle,draw,inner sep=1.1pt] (\name m) at (7.0,\y) {};
        \node[right] at (7.25,\y) {\rightlab};

        \draw (\name 1)--(4.2,\y);
        \draw (4.7,\y)--(\name 2);
        \draw (\name 2)--(5.55,\y);
        \draw (6.25,\y)--(\name m);
    }

\draw[densely dashed] (ui) to[bend left=12] (a1);
\draw[densely dashed] (ui) to[bend left=2] (a2);
\draw[densely dashed] (ui) to[bend right=10] (ahi);
    \node[font=\scriptsize,align=center] at (4.05,1.45) {};
\end{tikzpicture}
\caption{A schematic of the simple graph \(G_{\boldsymbol w}\).  The ``anchor'' star allows us to replace
parallel root-edges from the multigraph construction: a path vertex joined to
\(a_1,\dots,a_h\) serves as a replacement for \(h\) parallel edges to the root. Twin copies of the paths allow determinant factors to survive even in the presence of anchor vertices.}
\label{fig:simple-graph}
\end{figure}

\medskip

The simple graph constructed above has
\[
N_{m,q}=4m+q-1
\]
vertices.  In the final argument \(q\) will be tied to \(m\), namely
\(q=\lfloor \eps m\rfloor\), so these graphs do not realize every
integer value of \(n\).  However, this can be fixed by a simple padding argument.  If \(n\ge N_{m,q}\), let \(G_{\boldsymbol w}^{(n)}\) be obtained from
\(G_{\boldsymbol w}\) by adding
\[
\ell=n-N_{m,q}
\]
new vertices \(z_1,\dots,z_\ell\), each adjacent only to the fixed anchor vertex 
\(a_1\).  This padding does not change the number of spanning
trees.  Indeed, each edge \(a_1z_j\) is a bridge and hence belongs to
every spanning tree; deleting all leaves $z_1,\dots, z_\ell$ and their forced edges gives
a spanning tree of \(G_{\boldsymbol w}\), and conversely every spanning tree
of \(G_{\boldsymbol w}\) extends uniquely by adding these edges.
Therefore
\begin{equation}\label{eq:padding-preserves-tau}
\tau(G_{\boldsymbol w}^{(n)})=\tau(G_{\boldsymbol w}).
\end{equation}

\subsection{The determinant factor}\label{sec:determinant}

In the multigraph construction, deleting the root \(\rho\) left two decoupled
path blocks, and hence gave the exact identity (see \cref{prop:multigraph-warmup})
\[
\tau(H_{\boldsymbol w})
=
K_m(w_1,\dots,w_m)K_{m-1}(w_1,\dots,w_{m-1}).
\]
For the simple graph \(G_{\boldsymbol w}\), the corresponding equality is no
longer true: after deleting \(\rho\), the anchor vertices \(a_1,\dots,a_q\)
remain, so the path blocks are still coupled to the rest of the cofactor.
The point of the twin copies is that the same continuant product nevertheless
survives as an exact divisor of the number of spanning trees. 

\begin{proposition}
\label{prop:simple-graph-determinant-factor}
For every \(
\boldsymbol w=(w_1,\dots,w_m)\in\set{2,3,\dots,q+1}^m
\),
\[
D_{\boldsymbol w}:=
K_m(w_1,\dots,w_m)K_{m-1}(w_1,\dots,w_{m-1}) \text{ divides } \tau(G_{\boldsymbol w}).
\]
\end{proposition}

\begin{proof}
Delete from the Laplacian of \(G_{\boldsymbol w}\) the row and column indexed
by the root \(\rho\), and denote the resulting Matrix-Tree cofactor by
\(M_{\boldsymbol w}\).  Set
\[
A_{\boldsymbol w}=T_m(w_1,\dots,w_m),
    \qquad
B_{\boldsymbol w}=T_{m-1}(w_1,\dots,w_{m-1}).
\]
The two \(m\)-vertex paths contribute identical copies of
\(A_{\boldsymbol w}\), and the two \((m-1)\)-vertex paths contribute identical
copies of \(B_{\boldsymbol w}\).  The remaining path-to-anchor and
anchor block entries are encoded by integer matrices
\[
P_{\boldsymbol w}\in\Z^{m\times q},\qquad
Q_{\boldsymbol w}\in\Z^{(m-1)\times q},\qquad
F_{\boldsymbol w}\in\Z^{q\times q}.
\]
With the vertex order
\[
u_1,\dots,u_m,\ 
v_1,\dots,v_m,\ 
u'_1,\dots,u'_{m-1},\ 
v'_1,\dots,v'_{m-1},\ 
a_1,\dots,a_q,
\]
the cofactor has block form
\begin{equation}\label{eq:reduced-laplacian-block}
M_{\boldsymbol w}=
\begin{pmatrix}
A_{\boldsymbol w}&0&0&0&P_{\boldsymbol w}\\
0&A_{\boldsymbol w}&0&0&P_{\boldsymbol w}\\
0&0&B_{\boldsymbol w}&0&Q_{\boldsymbol w}\\
0&0&0&B_{\boldsymbol w}&Q_{\boldsymbol w}\\
P_{\boldsymbol w}^T&P_{\boldsymbol w}^T&
Q_{\boldsymbol w}^T&Q_{\boldsymbol w}^T&F_{\boldsymbol w}
\end{pmatrix}.
\end{equation}
By \cref{thm:matrix-tree},
\[
\tau(G_{\boldsymbol w})=\det M_{\boldsymbol w}.
\]
Applying \cref{lem:two-copy} to the first two copies of \(A_{\boldsymbol w}\)
in \eqref{eq:reduced-laplacian-block} gives
\begin{equation}\label{eq:after-first-peel}
\det M_{\boldsymbol w}
=
\det(A_{\boldsymbol w})
\det
\begin{pmatrix}
A_{\boldsymbol w}&0&0&P_{\boldsymbol w}\\
0&B_{\boldsymbol w}&0&Q_{\boldsymbol w}\\
0&0&B_{\boldsymbol w}&Q_{\boldsymbol w}\\
2P_{\boldsymbol w}^T&Q_{\boldsymbol w}^T&
Q_{\boldsymbol w}^T&F_{\boldsymbol w}
\end{pmatrix}.
\end{equation}
In the remaining determinant, simultaneously permuting block rows and block
columns to put the two \(B_{\boldsymbol w}\) blocks first does not change the determinant, so that
\[
\det
\begin{pmatrix}
A_{\boldsymbol w}&0&0&P_{\boldsymbol w}\\
0&B_{\boldsymbol w}&0&Q_{\boldsymbol w}\\
0&0&B_{\boldsymbol w}&Q_{\boldsymbol w}\\
2P_{\boldsymbol w}^T&Q_{\boldsymbol w}^T&
Q_{\boldsymbol w}^T&F_{\boldsymbol w}
\end{pmatrix}
=
\det
\begin{pmatrix}
B_{\boldsymbol w}&0&0&Q_{\boldsymbol w}\\
0&B_{\boldsymbol w}&0&Q_{\boldsymbol w}\\
0&0&A_{\boldsymbol w}&P_{\boldsymbol w}\\
Q_{\boldsymbol w}^T&Q_{\boldsymbol w}^T&
2P_{\boldsymbol w}^T&F_{\boldsymbol w}
\end{pmatrix}.
\]
A second application of \cref{lem:two-copy} therefore gives
\[
\det
\begin{pmatrix}
B_{\boldsymbol w}&0&0&Q_{\boldsymbol w}\\
0&B_{\boldsymbol w}&0&Q_{\boldsymbol w}\\
0&0&A_{\boldsymbol w}&P_{\boldsymbol w}\\
Q_{\boldsymbol w}^T&Q_{\boldsymbol w}^T&
2P_{\boldsymbol w}^T&F_{\boldsymbol w}
\end{pmatrix}
=
\det(B_{\boldsymbol w})
\det
\underbrace{\begin{pmatrix}
B_{\boldsymbol w}&0&Q_{\boldsymbol w}\\
0&A_{\boldsymbol w}&P_{\boldsymbol w}\\
2Q_{\boldsymbol w}^T&2P_{\boldsymbol w}^T&F_{\boldsymbol w}
\end{pmatrix}}_{R_{\boldsymbol w}}.
\]
Therefore,
\begin{equation}\label{eq:det-factorization}
\det M_{\boldsymbol w}
=
\det(A_{\boldsymbol w})\det(B_{\boldsymbol w})\det R_{\boldsymbol w},
\end{equation}
where $\det R_{\boldsymbol w} \in \mb{Z}$ since $R_{\boldsymbol w}$ is an integer matrix. The conclusion now follows since
\[
\det(A_{\boldsymbol w})=K_m(w_1,\dots,w_m),
    \qquad
\det(B_{\boldsymbol w})=K_{m-1}(w_1,\dots,w_{m-1}). \qedhere
\]
\end{proof}

\section{Proof of \cref{thm:main}}\label{sec:counting}

We now complete the proof by counting. For convenience of notation, for $m\geq 3$ and $q\geq 1$, let
\[\mc{W}_{m,q} := \{2,\dots, q+1\}^m.\]
The determinant factorization (\cref{prop:simple-graph-determinant-factor}) assigns
to each \(\boldsymbol w \in \mc{W}_{m,q}\) a divisor
\[
D_{\boldsymbol w} := K_m(w_1,\dots, w_m)K_{m-1}(w_1,\dots, w_{m-1})\mid \tau(G_{\boldsymbol w}).
\]
We first use the injectivity of the continuant pair map (\cref{lem:continuant-reconstruction}) and the divisor bound to show that  the integers \(D_{\boldsymbol w}\) take many distinct
values, and then use the divisor bound once more to show that these many
divisors cannot be supported on too few spanning-tree counts.

\begin{proposition}\label{prop:tree-count-lower-mq}
Fix $\eta>0$.  For all sufficiently large $m+q$, with $m\ge3$ and $q\ge1$, 
\[
\abs{\set{\tau(G_{\boldsymbol{w}}):\boldsymbol{w}\in\mc W_{m,q}}}
\ge
\frac{q^m}{(q+1)^{2\eta m}N^{\eta N}},
\]
where $N = N_{m,q} = 4m+q-1$.
\end{proposition}

\begin{proof}
We first show that the divisors \(D_{\boldsymbol w}\) themselves
take many distinct values.  By definition,
\[
D_{\boldsymbol w}
=
K_m(w_1,\dots,w_m)K_{m-1}(w_1,\dots,w_{m-1}).
\]
By \cref{lem:continuant-reconstruction}, the two factors in this product,
taken as an ordered pair, determine the word \(\boldsymbol w\).  Hence, for a
fixed integer \(D\), the number of vectors $\boldsymbol{w} \in \mc{W}_{m,q}$ with \(D_{\boldsymbol w}=D\) is at
most the number of ordered factorizations
\[
D=AB,
\]
namely at most \(\sigma_0(D)\). Moreover, by
\cref{lem:continuant-product-bound},
\[
D_{\boldsymbol w}\le (q+1)^m(q+1)^{m-1}\le (q+1)^{2m}.
\]
Applying the divisor bound (\cref{lem:divisor-bound}) with \(X=(q+1)^{2m}\), we get
\[
\sigma_0(D)\le (q+1)^{2\eta m}
\]
for all sufficiently large \(m+q\). Thus every fiber of the map
\[
\boldsymbol w\mapsto D_{\boldsymbol w}
\]
has size at most \((q+1)^{2\eta m}\).  Since
\(|\mc W_{m,q}|=q^m\), it follows that
\begin{equation}\label{eq:number-distinct-D}
\abs{\set{D_{\boldsymbol w}:\boldsymbol w\in\mc W_{m,q}}}
\ge
\frac{q^m}{(q+1)^{2\eta m}}.
\end{equation}

Now choose one representative word \(\boldsymbol w_D\) for each distinct value
\(D\) counted in \eqref{eq:number-distinct-D}, and set
\[
T(D):=\tau(G_{\boldsymbol w_D}).
\]
By \cref{prop:simple-graph-determinant-factor}, \(D\) is a positive divisor of
\(T(D)\).  Hence, for any fixed integer \(T\), there are at most
\(\sigma_0(T)\) values of \(D\) with \(T(D)=T\).

We bound this multiplicity uniformly.  Since every \(G_{\boldsymbol w}\) has
\(N\) vertices, every spanning tree of \(G_{\boldsymbol w}\) is a spanning tree
of \(K_N\), and Cayley's formula \cite{cayley} gives
\[
T(D)\le \tau(K_N)=N^{N-2}\le N^N.
\]
Applying \cref{lem:divisor-bound} with \(X=N^N\), we have
\[
\sigma_0(T(D))\le N^{\eta N}
\]
for all sufficiently large \(m+q\).  Combining this with
\eqref{eq:number-distinct-D}, the number of distinct spanning-tree counts is
at least
\[
\frac{q^m}{(q+1)^{2\eta m}N^{\eta N}}.
\]
This proves the proposition.
\end{proof}

We now take \(q=\lfloor \eps m\rfloor\), with \(0<\eps<1\) fixed.  Then
\(N=4m+q-1=(4+\eps)m+O(1)\), and
\[
\log q,\ \log(q+1),\ \log N=\log m+O_\eps(1).
\]
Thus \cref{prop:tree-count-lower-mq} gives, for every fixed \(\eta>0\) and all
sufficiently large \(m\),
\begin{equation}\label{eq:asymptotic-m}
\log\abs{\set{\tau(G_{\boldsymbol w}):\boldsymbol w\in\mc W_{m,q}}}
\ge
\paren{1-(6+\eps)\eta}m\log m-O_{\eps,\eta}(m).
\end{equation}

Finally, we finish the proof of the main theorem.

\begin{proof}[Proof of \cref{thm:main}]
 Fix \(0<c<1/4\), and choose \(\eps\in(0,1)\) and
\(\eta>0\) so that
\[
\alpha:=\frac{1-(6+\eps)\eta}{4+\eps}>c.
\]
For large \(n\), choose \(m\) maximal such that
\[
N_m:=4m+\lfloor\eps m\rfloor-1\le n.
\]
Then \(m\to\infty\) with \(n\), and maximality gives
\[
n<N_{m+1}\le (4+\eps)(m+1)-1,
\]
so
\begin{equation}\label{eq:mlogm-vs-nlogn}
m\log m\ge \frac{1}{4+\eps}n\log n-O_\eps(n).
\end{equation}

Using the padded graphs \(G_{\boldsymbol w}^{(n)}\), whose spanning-tree counts
agree with those of \(G_{\boldsymbol w}\) (see
\eqref{eq:padding-preserves-tau}), \eqref{eq:asymptotic-m} and
\eqref{eq:mlogm-vs-nlogn} imply
\[
\log\abs{ \mc{T}_n}
\ge
\alpha n\log n-O_{\eps,\eta}(n) \geq cn \log n
\]
for all sufficiently large
\(n\). 
\end{proof}

\bibliography{main}
\bibliographystyle{amsplain0}

\end{document}